\documentclass[a4paper,12pt]{article}

\usepackage{amssymb}
\usepackage{amsmath}
\usepackage{amsthm}
\usepackage{amsfonts}
\usepackage[active]{srcltx}
\newtheorem{theorem}{Theorem}

\newtheorem{corollary}[theorem]{Corollary}

\newtheorem{definition}[theorem]{Definition}
\newtheorem{example}[theorem]{Example}

\newtheorem{lemma}[theorem]{Lemma}

\newtheorem{remark}[theorem]{Remark}

\renewenvironment{proof}[1][Proof]{\noindent\textbf{#1.} }{\hfill$\square$}
\begin{document}

\date{December 9, 2013}
\title{{\bf Robust Solutions to Multi-Objective Linear Programs with Uncertain
Data}\thanks{%
This research was partially supported by the MICINN of Spain, Grant
MTM2011-29064-C03-02, and the Australian Research Council, Grant
DP120100467. }}
\author{M.A. Goberna\thanks{%
Corresponding author. Tel.: +34 965903533. Fax: +34 965903531.} \thanks{%
Dept. of Statistics and Operations Research, Alicante University, 03071
Alicante, Spain.}, V. Jeyakumar\thanks{%
Dept. of Applied Mathematics, University of New South Wales, Sydney 2052,
Australia. }, G. Li\footnotemark[4], and J. Vicente-P\'{e}rez\footnotemark[4]
}
\maketitle

\begin{abstract}
In this paper we examine multi-objective linear programming problems in the
face of data uncertainty both in the objective function and the constraints.
First, we derive a formula for radius of robust feasibility guaranteeing
constraint feasibility for all possible uncertainties within a specified
uncertainty set under affine data parametrization. We then present a
complete characterization of robust weakly effcient solutions that are
immunized against rank one objective matrix data uncertainty. We also
provide classes of commonly used constraint data uncertainty sets under
which a robust feasible solution of an uncertain multi-objective linear
program can be numerically checked whether or not it is a robust weakly
efficient solution.
\end{abstract}

\let\thefootnote\relax\footnotetext{%
E-mail addresses: mgoberna@ua.es (M.A. Goberna), v.jeyakumar@unsw.edu.au (V.
Jeyakumar), g.li@unsw.edu.au (G. Li), jose.vicente@ua.es (J.
Vicente-P\'erez).}


\section{Introduction}

\label{sec:1}

Consider the deterministic multi-objective linear programming problem
\begin{equation*}
(\overline{P})\quad \text{V-}\min \left\{ \overline{C}x:\overline{a}%
_{j}^{\top }x\geq \overline{b}_{j},\,j=1,\ldots ,p\right\}
\end{equation*}
where V-$\min $ stands for \emph{vector minimization}, $\overline{C}$\ is a
real $m\times n$ matrix called \emph{objective matrix}, $x\in $ $\mathbb{R}%
^{n}$ is the decision variable, and $\left( \overline{a}_{j},\overline{b}%
_{j}\right) \in \mathbb{R}^{n+1}$, for $\,j=1,\ldots ,p$, are the constraint
input data of the problem. The problem $(\overline{P})$ has been extensively
studied in the literature (see e.g. the overviews \cite{BDMS08} and \cite%
{Eh05}), where perfect information is often assumed (that is, accurate
values for the input quantities or parameters), despite the reality that
such precise knowledge is rarely available in practice for real-world
optimization problems.

The data of real-world optimization problems are often uncertain (that is,
they are not known exactly at the time of the decision) due to estimation
errors, prediction errors or lack of information. Scalar uncertain
optimization problems have been traditionally treated via sensitivity
analysis, which estimates the impact of small perturbations of the data in
the optimal value, while robust optimization, which provides a deterministic
framework for uncertain problems (\cite{BEN09},\cite{EL97}), has recently
emerged as a powerful alternative approach.

Particular types of uncertain multi-objective linear programming problems
have been studied, e.g. \cite{Sitarz08} considers changes in one objective
function via sensitivity analysis, \cite{GDH12} and \cite{PLP13} consider
changes in the whole objective function $x\mapsto \overline{C}x$, and \cite%
{GJLP} deals with changes in the constraints, the latter three works using
different robustness approaches. The purpose of the present work is to study
multi-objective linear programming problems in the face of data uncertainty
both in the objective function and constraints from a robustness perspective.

\smallskip

Following the robust optimization framework, the multi-objective problem $(%
\overline{P})$ in the face of \emph{data uncertainty} both in the objective
matrix and in the data of the constraints can be captured by a parameterized
multi-objective linear programming problem of the form
\begin{equation*}
(P_{w})\quad \text{V-}\min \left\{ Cx:a_{j}^{\top }x\geq b_{j},\,j=1,\ldots
,p\right\}
\end{equation*}%
where the input data, the rows of $C$ and $\left( a_{j},b_{j}\right) ,$ $%
j=1,\ldots ,p,$ are uncertain vectors. The sets $\mathcal{U}$ and $\mathcal{V%
}_{j}$, $j=1,\ldots,p$, are specified uncertainty sets that are bounded, but
often infinite sets, $C\in \mathcal{U}\subset \mathbb{R}^{m\times n}$ and $%
(a_{j},b_{j})\in \mathcal{V}_{j}\subset \mathbb{R}^{n+1},j=1,\ldots ,p$. So,
the uncertain parameter is $w:=(C,(a_{1},b_{1}),...,(a_{p},b_{p}))\in
\mathcal{W}:=\mathcal{U}\times \prod_{j=1}^{p}\mathcal{V}_{j}$. By enforcing
the constraints for all possible uncertainties within $\mathcal{V}%
_{j},\,j=1,\ldots ,p$, the uncertain problem becomes the following uncertain
multi-objective linear semi-infinite programming problem
\begin{equation*}
(P_{C})\quad \text{V-}\min \left\{ Cx:a_{j}^{\top }x\geq b_{j},\;\forall
(a_{j},b_{j})\in \mathcal{V}_{j},\,j=1,\ldots ,p\right\}
\end{equation*}%
where the data uncertainty occurs only in the objective function and
\begin{equation*}
X:=\{x\in \mathbb{R}^{n}:a_{j}^{\top }x\geq b_{j},\;\forall (a_{j},b_{j})\in
\mathcal{V}_{j},\,j=1,\ldots ,p\},
\end{equation*}%
is the robust feasible set of $(P_{w})$.

Following the recent work on robust linear programming (see \cite{BEN09}),
some of the key questions of multi-objective linear programming under data
uncertainty include:

\begin{itemize}
\item[I.] (\emph{Guaranteeing robust feasibility}) How to guarantee
non-emptiness of the robust feasible set $X$ for specified uncertainty sets $%
\mathcal{V}_j, \, j=1,\ldots,p$?

\item[II.] (\emph{Defining and identifying robust solutions}) How to define
and characterize a robust solution that is immunized against data
uncertainty for the uncertain multi-objective problem $(P_C)$?

\item[III.] (\emph{Numerical tractability of robust solutions}) For what
classes of uncertainty sets robust solutions can be numerically checked?
\end{itemize}

In this paper, we provide some answers to the above questions for the
multi-objective linear programming problem $(P_{C})$ in the face of data
uncertainty. In particular, we derive a formula for the radius of robust
feasibility guaranteeing non-emptiness of the robust feasible set $X$ of $%
(P_{C})$ under affinely parameterized data uncertainty. Then, we establish
complete characterizations of robust weakly efficient solutions under rank
one objective matrix data uncertainty (the same type of uncertainty
considered in \cite{PLP13} for efficient solutions of similar problems with
deterministic constraints). We finally provide classes of commonly used
uncertainty sets under which robust feasible solutions can be numerically
checked whether or not they are robust weakly efficient solutions. \smallskip

\section{Radius of robust feasibility}

\label{sec:2}

In this section, we first discuss the feasibility of our uncertain
multi-objective model under affine constraint data perturbations. In other
words, for any given matrix $\overline{C}\in \mathbb{R}^{m\times n}$, we
study the feasibility of the problem
\begin{equation*}
\begin{array}{lrl}
(P_{\alpha }) & \text{V-}\min & \overline{C}x \\
& \text{s.t.} & a_{j}^{\top }x\geq b_{j},\,\forall (a_{j},b_{j})\in \mathcal{%
V}_{j}^{\alpha },\,j=1,\ldots ,p,%
\end{array}%
\end{equation*}%
for $\alpha \geq 0,$\ where the uncertain set-valued mapping $\mathcal{V}%
_{j}^{\alpha }$ takes the form
\begin{equation}
\mathcal{V}_{j}^{\alpha }:=(\overline{a}_{j},\overline{b}_{j})+\alpha
\mathbb{B}_{n+1},\quad j=1,\ldots ,p,  \label{uta}
\end{equation}%
with $\{x\in \mathbb{R}^{n}:\overline{a}_{j}^{\top }x\geq \overline{b}%
_{j}\}\neq \emptyset $, and $\mathbb{B}_{n+1}$ denotes the closed unit ball
for the Euclidean norm $\left\Vert \cdot \right\Vert $ in $\mathbb{R}^{n+1}$.

Let $\mathcal{V}:=\prod_{j=1}^{p}\mathcal{V}_{j}$. The \emph{radius of
feasibility} associated with $\mathcal{V}_{j}$, $j=1,\ldots ,p,$ as in %
\eqref{uta} is defined to be
\begin{equation}
\rho (\mathcal{V}):=\sup \left\{ \alpha \in \mathbb{R}_{+}:(P_{\alpha })%
\text{ is feasible for }\alpha \right\} .  \label{4.4}
\end{equation}

To establish the formula for the radius of robust feasibility, we first note
a known useful characterization of feasibility of an infinite inequality
system in terms of the closure of the convex cone generated by its set of
coefficient vectors.

\begin{lemma}[{\protect\cite[Theorem 4.4]{GL98}}]
\label{lemma:16} Let $T$ be an arbitrary index set. Then, $\{x\in \mathbb{R}%
^{n}:a_{t}^{\top }x\geq b_{t},t\in T\}\neq \emptyset $ if and only if $%
(0_{n},1)\notin \mathrm{cl}\,\mathrm{cone}\{(a_{t},b_{t}):t\in T\}$.
\end{lemma}

Using the above Lemma, we first observe that the radius of robust
feasibility $\rho (\mathcal{V})$ is a non-negative number since, given $%
j=1,...,p$, $(0_{n},1)\in (\overline{a}_{j},\overline{b}_{j})+\alpha \mathbb{%
B}_{n+1}$ for a positive large enough $\alpha $, in which case the
corresponding problem $(P_{\alpha })$ is not feasible.

The next result provides a formula for the radius of feasibility which
involves the so-called \emph{hypographical set} (\cite{CLPT05}) of the
system $\{\overline{a}_{j}^{\top }x\geq \overline{b}_{j},j=1,\ldots ,p\}$,
defined as
\begin{equation}
H(\overline{a},\overline{b}):={\rm conv}\left\{ (\overline{a}_{j},\overline{%
b}_{j}),j=1,\ldots ,p\right\} +\mathbb{R}_{+}\left\{ \left( 0_{n},-1\right)
\right\} ,  \label{epi}
\end{equation}%
where $\overline{a}:=(\overline{a}_{1},\ldots ,\overline{a}_{p})\in (\mathbb{%
R}^{n})^{p}$ and $\overline{b}:=(\overline{b}_{1},\ldots ,\overline{b}%
_{p})\in \mathbb{R}^{p}$. We observe that $H(\overline{a},\overline{b})$ is
the sum of the polytope ${\rm conv}\left\{ (\overline{a}_{j},\overline{b}%
_{j}),j=1,\ldots ,p\right\} $ with the closed half-line $\mathbb{R}%
_{+}\left\{ \left( 0_{n},-1\right) \right\} ,$ so that it is a polyhedral
convex set.

\begin{lemma}
\label{lemma:6} Let $(\overline{a}_{j},\overline{b}_{j})\in \mathbb{R}%
^{n}\times \mathbb{R}$, $j=1,\ldots ,p$, and $\alpha \geq 0$. Suppose that
\begin{equation*}
(0_{n},1)\in {\rm cl}{\rm cone}\left( \{(\overline{a}_{j},\overline{b}%
_{j}),j=1,...,p\}+\alpha \mathbb{B}_{n+1}\right) .
\end{equation*}%
Then, for all $\delta >0$, we have
\begin{equation*}
(0_{n},1)\in {\rm cone}\left( \{(\overline{a}_{j},\overline{b}%
_{j}),j=1,...,p\}+(\alpha +\delta )\mathbb{B}_{n+1}\right) .
\end{equation*}
\end{lemma}

\begin{proof}
Let $\delta >0$. To see the conclusion, we assume by contradiction that
\begin{equation*}
(0_{n},1)\notin {\rm cone}\left( \{(\overline{a}_{j},\overline{b}%
_{j}),j=1,...,p\}+(\alpha +\delta )\mathbb{B}_{n+1}\right) .
\end{equation*}%
Then, the separation theorem implies that there exists $(\xi ,r)\in \mathbb{R%
}^{n+1}\backslash \{0_{n+1}\}$ such that for all $(y,s)\in {\rm cone}\left(
\{(\overline{a}_{j},\overline{b}_{j}),j=1,...,p\}+(\alpha +\delta )\mathbb{B}%
_{n+1}\right) $ one has
\begin{equation}
r=\langle (\xi ,r),(0_{n},1)\rangle \leq 0\leq \langle (\xi ,r),(y,s)\rangle
,  \label{useful}
\end{equation}%
where $\langle \cdot ,\cdot \rangle $ denotes the usual inner product, i.e. $%
\langle (\xi ,r),(y,s)\rangle =\xi ^{\top }y+rs$. Recall that $(0_{n},1)\in
{\rm cl}{\rm cone}\left( \{(\overline{a}_{j},\overline{b}%
_{j}),j=1,...,p\}+\alpha \mathbb{B}_{n+1}\right) $. So, there exist
sequences $\{(y_{k},s_{k})\}_{k\in \mathbb{N}}\subset \mathbb{R}^{n}\times
\mathbb{R},$ $\{ \mu _{k}^{j}\} _{k\in \mathbb{N}}\subset \mathbb{R}_{+},$
and $\{ (z_{k}^{j},t_{k}^{j})\}_{k\in \mathbb{N}}\subset \mathbb{B}_{n+1}$, $%
j=1,\ldots ,p$, such that $(y_{k},s_{k})\rightarrow (0_{n},1)$ and
\begin{equation*}
(y_{k},s_{k})=\sum_{j=1}^{p}{\mu _{k}^{j}\left( (\overline{a}_{j},\overline{b%
}_{j})+\alpha (z_{k}^{j},t_{k}^{j})\right) }.
\end{equation*}%
If $\{\sum_{j=1}^{p}\mu _{k}^{j}\}_{k\in \mathbb{N}}$ is a bounded sequence,
by passing to subsequence if necessary, we have
\begin{equation*}
(0_{n},1)\in {\rm cone}\left( \{(\overline{a}_{j},\overline{b}%
_{j}),j=1,...,p\}+\alpha \mathbb{B}_{n+1}\right) .
\end{equation*}%
Thus, the claim is true whenever $\{\sum_{j=1}^{p}\mu _{k}^{j}\}_{k\in
\mathbb{N}}$ is a bounded sequence. So, we may assume that $%
\sum_{j=1}^{p}\mu _{k}^{j}\rightarrow +\infty $ as $k\rightarrow \infty $.
Let $(y,s)\in \mathbb{B}_{n+1}$ be such that $\langle (y,s),(\xi ,r)\rangle
=\Vert (\xi ,r)\Vert $. Note that
\begin{equation*}
\sum_{j=1}^{p}{\mu _{k}^{j}\left( (\overline{a}_{j},\overline{b}_{j})+\alpha
(z_{k}^{j},t_{k}^{j})-\delta (y,s)\right) }\in {\rm cone}\left( \{(%
\overline{a}_{j},\overline{b}_{j}),j=1,...,p\}+(\alpha +\delta )\mathbb{B}%
_{n+1}\right) .
\end{equation*}%
Then, \eqref{useful} implies that
\begin{eqnarray*}
r \ \leq \ 0 &\leq &\langle (\xi ,r),\sum_{j=1}^{p}{\mu _{k}^{j}\left( (%
\overline{a}_{j},\overline{b}_{j})+\alpha z_{k}^{j}\right) }\rangle
-(\sum_{j=1}^{p}\mu _{k}^{j})\,\delta \Vert (\xi ,r)\Vert \medskip \\
&=&\langle (\xi ,r),(y_{k},s_{k})\rangle -(\sum_{j=1}^{p}\mu
_{k}^{j})\,\delta \Vert (\xi ,r)\Vert .
\end{eqnarray*}%
Passing to the limit, we arrive to a contradiction as $(\xi ,r)\neq 0_{n+1}$%
, $\delta >0$, $\sum_{j=1}^{p}\mu _{k}^{j}\rightarrow +\infty $ and $%
(y_{k},s_{k})\rightarrow (0_{n},1)$.
\end{proof}

\smallskip

We now provide our promised formula for the radius of robust feasibility.
Observe that, since $0_{n+1}\notin H(\overline{a},\overline{b})$ by Lemma %
\ref{lemma:16},\ $d\left( 0_{n+1},H(\overline{a},\overline{b})\right) $ can
be computed minimizing $\left\Vert \cdot \right\Vert ^{2}$\ on $H(\overline{a%
},\overline{b})$ (i.e. by solving a convex quadratic program).

\begin{theorem}[Radius of robust feasibility]
\label{PropRobustFeasibility} For $(P_\alpha )$, let $(\overline{a}_{j},\overline{b}_{j})\in
\mathbb{R}^{n}\times \mathbb{R}$, $j=1,\ldots ,p,$ with $\{x\in \mathbb{R}%
^{n}:\overline{a}_{j}^{\top }x\geq \overline{b}_{j},j=1,...,p\}\neq
\emptyset $. Let $\mathcal{V}_{j}:=(\overline{a}_{j},\overline{b}%
_{j})+\alpha \mathbb{B}_{n+1},\ j=1,\ldots ,p,$ and $\mathcal{V}%
:=\prod_{j=1}^{p}\mathcal{V}_{j}$. Let $\rho (\mathcal{V})$ be the radius of
robust feasibility as given in \eqref{4.4} and let $H(\overline{a},\overline{%
b})$ be the hypographical set as given in \eqref{epi}. Then, $\rho (\mathcal{%
V})=d\left( 0_{n+1},H(\overline{a},\overline{b})\right) $.
\end{theorem}

\begin{proof}
If a given $(v,w)\in (\mathbb{R}^{n})^{p}\times \mathbb{R}^{p}$ is
interpreted as a perturbation of $(\overline{v},\overline{w})\in (\mathbb{R}%
^{n})^{p}\times \mathbb{R}^{p}$, we can measure the size of this
perturbation as the supremum of the distances between the vectors of
coefficients corresponding to the same index. This can be done by endowing
the parameter space $(\mathbb{R}^{n})^{p}\times \mathbb{R}^{p}$ with the
metric $\widetilde{d}$ defined by
\begin{equation*}
\widetilde{d}\left( (v,w),(p,q)\right) :=\sup\limits_{j=1,...,p}\left\Vert
(v_{j},w_{j})-(p_{j},q_{j})\right\Vert ,\mbox{ for }(v,w),(p,q)\in (\mathbb{R%
}^{n})^{p}\times \mathbb{R}^{p}.
\end{equation*}%
Let $\overline{a}\in (\mathbb{R}^{n})^{p}$ and $\overline{b}\in \mathbb{R}%
^{p}$ be as in \eqref{epi}. Denote the set consisting of all inconsistent
parameters by $\Theta _{i}$, that is,
\begin{equation*}
\Theta _{i}=\{(v,w)\in (\mathbb{R}^{n})^{p}\times \mathbb{R}^{p}:\{x\in
\mathbb{R}^{n}:v_{j}^{\top }x\geq w_{j},j=1,...,p\}=\emptyset \}.
\end{equation*}%
We now show that
\begin{equation}
\widetilde{d}\left( (\overline{a},\overline{b}),\Theta _{i}\right) =d\left(
0_{n+1},H(\overline{a},\overline{b})\right) .  \label{eq:lala}
\end{equation}%
By Lemma \ref{lemma:16}, $d\left( 0_{n+1},H(\overline{a},\overline{b}%
)\right) >0.$\ Let $(a,b)\in H(\overline{a},\overline{b})$ be such that $%
\Vert (a,b)\Vert =d\left( 0_{n+1},H(\overline{a},\overline{b})\right) .$
Then, $0_{n+1}\in H_{1}$ where
\begin{equation*}
H_{1}:=H(\overline{a},\overline{b})-(a,b)={\rm conv}\left\{ (\overline{a}%
_{j}-a,\overline{b}_{j}-b),j=1,\ldots ,p\right\} +\mathbb{R}_{+}\left\{
\left( 0_{n},-1\right) \right\} .
\end{equation*}%
So, there exist $\lambda _{j}\geq 0$ with $\sum_{j=1}^{p}\lambda _{j}=1$ and
$\mu \geq 0$ such that
\begin{equation*}
0_{n+1}=\sum_{j=1}^{p}\lambda _{j}(\overline{a}_{j}-a,\overline{b}%
_{j}-b)+\mu (0_{n},-1).
\end{equation*}%
This shows that
\begin{equation*}
(0_{n},1)=\sum_{j=1}^{p}\tfrac{\lambda _{j}}{\mu +\frac{1}{k}}(\overline{a}%
_{j}-a,\overline{b}_{j}-b+\frac{1}{k}),\text{ }k\in \mathbb{N}.
\end{equation*}%
So, $\{x:\left( \overline{a}_{j}-a\right) ^{\top }x\geq \overline{b}_{j}-b+%
\frac{1}{k},j=1,\ldots ,p\}=\emptyset $. Thus, $(\overline{a}-a,\overline{b}%
-b+\frac{1}{k})\in \Theta _{i}$, and so, $(\overline{a}-a,\overline{b}-b)\in
\mathrm{cl}\,\Theta _{i}$. It follows that
\begin{equation*}
\widetilde{d}\left( (\overline{a},\overline{b}),\Theta _{i}\right) =%
\widetilde{d}\left( (\overline{a},\overline{b}),\mathrm{cl}\,\Theta
_{i}\right) \leq \Vert (a,b)\Vert =d\left( 0_{n+1},H(\overline{a},\overline{b%
})\right) .
\end{equation*}%
To see \eqref{eq:lala}, we suppose on the contrary that $d\left( (\overline{a%
},\overline{b}),\Theta _{i}\right) <d\left( 0_{n+1},H(\overline{a},\overline{%
b})\right) .$ Then, there exist $\varepsilon _{0}>0,$ with $\varepsilon
_{0}<\Vert (a,b)\Vert ,$\ and $(\hat{a},\hat{b})\in \mathrm{bd}\,\Theta _{i}$
such that $\widetilde{d}\left( (\overline{a},\overline{b}),(\hat{a},\hat{b}%
)\right) =\widetilde{d}\left( (\overline{a},\overline{b}),\Theta _{i}\right)
<\Vert (a,b)\Vert -\varepsilon _{0}$. Then, one can find $\{ (\hat{a}^{k},%
\hat{b}^{k})\} _{k\in \mathbb{N}}\subset \Theta _{i}$ such that $(\hat{a}%
^{k},\hat{b}^{k})\rightarrow (\hat{a},\hat{b})$. So, Lemma \ref{lemma:16}
gives us that
\begin{equation*}
(0_{n},1)\in \mathrm{cl}\,\mathrm{cone}\{(\hat{a}_{j}^{k},\hat{b}%
_{j}^{k}):j=1,...,p\}=\mathrm{cone}\{(\hat{a}_{j}^{k},\hat{b}%
_{j}^{k}):j=1,...,p\}.
\end{equation*}%
Thus, there exist $\lambda _{j}^{k}\geq 0$ such that $(0_{n},1)=%
\sum_{j=1}^{p}\lambda _{j}^{k}(\hat{a}_{j}^{k},\hat{b}_{j}^{k})$. Note that $%
\sum_{j=1}^{p}\lambda _{j}^{k}>0$, and so,
\begin{equation*}
0_{n+1}=\sum_{j=1}^{p}\tfrac{\lambda _{j}^{k}}{\sum_{j=1}^{p}\lambda _{j}^{k}%
}(\hat{a}_{j}^{k},\hat{b}_{j}^{k})+\tfrac{1}{\sum_{j=1}^{p}\lambda _{j}^{k}}%
(0_{n},-1).
\end{equation*}%
Then as $k\rightarrow \infty $,
\begin{equation*}
\Vert \sum_{j=1}^{p}\tfrac{\lambda _{j}^{k}}{\sum_{j=1}^{p}\lambda _{j}^{k}}(%
\hat{a}_{j},\hat{b}_{j})+\tfrac{1}{\sum_{j=1}^{p}\lambda _{j}^{k}}%
(0_{n},-1)\Vert =\Vert \sum_{j=1}^{p}\tfrac{\lambda _{j}^{k}}{%
\sum_{j=1}^{p}\lambda _{j}^{k}}(\hat{a}_{j}-\hat{a}_{j}^{k},\hat{b}_{j}-\hat{%
b}_{j}^{k})\Vert \rightarrow 0.
\end{equation*}%
So, $0_{n+1}\in \mathrm{cl}\,H(\hat{a},\hat{b})=H(\hat{a},\hat{b}).$ It then
follows that there exist $\lambda _{j}\geq 0$ with $\sum_{j=1}^{p}\lambda
_{j}=1$ and $\mu \geq 0$ such that
\begin{equation*}
0_{n+1}=\sum_{j=1}^{p}\lambda _{j}(\hat{a}_{j},\hat{b}_{j})+\mu (0_{n},-1).
\end{equation*}%
Thus, we have%
\begin{equation*}
\begin{array}{ll}
& \Vert \sum_{j=1}^{p}\lambda _{j}(\overline{a}_{j},\overline{b}_{j})+\mu
(0,-1)\Vert \medskip \\
= & \Vert \big(\sum_{j=1}^{p}\lambda _{j}(\overline{a}_{j},\overline{b}%
_{j})+\mu (0,-1)\big)-\big(\sum_{j=1}^{p}\lambda _{j}(\hat{a}_{j},\hat{b}%
_{j})+\mu (0_{n},-1)\big)\Vert \medskip \\
= & \Vert \sum_{j=1}^{p}\lambda _{j}\big((\overline{a}_{j},\overline{b}%
_{j})-(\hat{a}_{j},\hat{b}_{j})\big)\Vert \medskip \\
\leq & \widetilde{d}\big((\overline{a},\overline{b}),(\hat{a},\hat{b})\big)%
<\Vert (a,b)\Vert -\varepsilon _{0},%
\end{array}%
\end{equation*}%
where the first inequality follows from the definition of $\widetilde{d}$
and $\lambda _{j}\geq 0$ with $\sum_{j=1}^{p}\lambda _{j}=1$. Note that $%
\sum_{j=1}^{p}\lambda _{j}(\overline{a}_{j},\overline{b}_{j})+\mu
(0_{n},-1)\in H(\overline{a},\overline{b})$. We see that $H(\overline{a},%
\overline{b})\cap (\Vert (a,b)\Vert -\varepsilon _{0})\mathbb{B}_{n+1}\neq
\emptyset $. This shows that $d(0_{n+1},H(\overline{a},\overline{b}))\leq
\Vert (a,b)\Vert -\varepsilon _{0}$ which contradicts the fact that $%
d(0_{n+1},H(\overline{a},\overline{b}))=\Vert (a,b)\Vert .$ Therefore, %
\eqref{eq:lala} holds.

Let $\alpha \in \mathbb{R}_{+}$ so that $(P_{C})$ is feasible for $\alpha $.
Then, $(a,b)\in \Theta _{i}$ implies that $\widetilde{d}\left( (\overline{a},%
\overline{b}),(a,b)\right) >\alpha $. Therefore, \eqref{eq:lala} gives us
that $d\left( 0_{n+1},H(\overline{a},\overline{b})\right) =\widetilde{d}%
\left( (\overline{a},\overline{b}),\Theta _{i}\right) \geq \alpha $. Thus, $%
\rho (\mathcal{V})\leq d\left( 0_{n+1},H(\overline{a},\overline{b})\right) $.

We now show that $\rho (\mathcal{V})=d\left( 0_{n+1},H(\overline{a},%
\overline{b})\right) $. To see this, we proceed by the method of
contradition and suppose that $\rho (\mathcal{V})<d\left( 0_{n+1},H(%
\overline{a},\overline{b})\right) $. The, there exists $\delta >0$ such that
$\rho (\mathcal{V})+2\delta <d\left( 0_{n+1},H(\overline{a},\overline{b}%
)\right) $. Let $\alpha _{0}:=\rho (\mathcal{V})+\delta .$ Then, by the
definition of $\rho (\mathcal{V}),$ $(P_{\alpha _{0}})$ is not feasible,
that is,
\begin{equation*}
\{x\in \mathbb{R}^{n}:a^{\top }x\geq b,\left( a,b\right) \in
\bigcup_{j=1}^{p}\left\{ (\overline{a}_{j},\overline{b}_{j})+\alpha \mathbb{B%
}_{n+1}\right\} \}=\emptyset .
\end{equation*}%
Hence, it follows from Lemma \ref{lemma:16} that%
\begin{equation*}
\left( 0_{n},1\right) \in {\rm cl}{\rm cone}\{\bigcup_{j=1}^{p}\left\{ (%
\overline{a}_{j},\overline{b}_{j})+\alpha \mathbb{B}_{n+1}\right\} \}.
\end{equation*}%
By applying Lemma \ref{lemma:6}, we can find $\mu _{j}\geq 0$ and $\left(
z_{j},t_{j}\right) \in \mathbb{B}_{n+1},$ $j=1,...,p,$ such that
\begin{equation*}
\left( 0_{n},1\right) =\sum_{j=1}^{p}\mu _{j}\left( (\overline{a}_{j},%
\overline{b}_{j})+\left( \alpha _{0}+\delta \right) \left(
z_{j},t_{j}\right) \right) .
\end{equation*}%
Let $\left( a_{j},b_{j}\right) =(\overline{a}_{j},\overline{b}_{j})+\left(
\alpha _{0}+\delta \right) \left( z_{j},t_{j}\right) ,$ $j=1,\ldots ,p,$ $%
a:=(a_{1},\ldots ,a_{p})\in (\mathbb{R}^{n})^{p}$ and $b:=(b_{1},\ldots
,b_{p})\in \mathbb{R}^{p}$. Then, $\widetilde{d}\left( (\overline{a},%
\overline{b}),(a,b)\right) \leq \alpha _{0}+\delta $ and
\begin{equation*}
\left( 0_{n},1\right) =\sum_{j=1}^{p}\mu _{j}\left( a_{j},b_{j}\right) \in
{\rm cone}\left\{ \left( a_{j},b_{j}\right) ,j=1,...,p\right\} .
\end{equation*}%
So, Lemma \ref{lemma:16} implies that $\left\{ x\in \mathbb{R}^{n}:\left(
a_{j},b_{j}\right) ,j=1,...,p\right\} =\emptyset $ and hence $(a,b)\in
\Theta _{i}.$ Thus,
\begin{equation*}
\widetilde{d}\left( (\overline{a},\overline{b}),\Theta _{i}\right) \leq
\widetilde{d}\left( (\overline{a},\overline{b}),(a,b)\right) \leq \alpha
_{0}+\delta =\rho (\mathcal{V})+2\delta .
\end{equation*}

Thus, from \eqref{eq:lala}, we see that $d\left( 0_{n+1},H(\overline{a},%
\overline{b})\right) \leq \widetilde{d}\left( (\overline{a},\overline{b}%
),\Theta _{i}\right) \leq \rho (\mathcal{V})+2\delta .$ This contradicts the
fact that $\rho (\mathcal{V})+2\delta <d\left( 0_{n+1},H(\overline{a},%
\overline{b})\right) .$ So, the conclusion follows.
\end{proof}

\begin{remark} We would like to note that we have given a self-contained and
simple proof for Theorem \ref{PropRobustFeasibility} by exploiting the finitness of
the linear inequality system. A semi-infinite version of
Theorem \ref{PropRobustFeasibility} under a regularity condition was presented
 in \cite[Theorem
3.3]{GJLP}, where the proof relies on several results in \cite{CLPT05} and
\cite{CLPT11}.
\end{remark}

In the following example we show how the radius of robust feasibility of $(P_{\alpha })$ can be calculated
using Theorem \ref{PropRobustFeasibility}.

\begin{example}{\bf (Calculating radius of robust feasibility)}
\label{ExampleRadius} Consider $(P_{\alpha })$ with $n=3,$ $p=5$ and $%
\mathcal{V}_{j}^{\alpha }\ $as in (\eqref{uta}), with
\begin{equation}
\left\{ (\overline{a}_{j},\overline{b}_{j}),j=1,\ldots ,5\right\} =\left\{
\left(
\begin{array}{c}
-2 \\
-1 \\
-2 \\
-6%
\end{array}%
\right) ,\left(
\begin{array}{c}
-1 \\
-2 \\
-2 \\
-6%
\end{array}%
\right) ,\left(
\begin{array}{c}
-1 \\
0 \\
0 \\
-3%
\end{array}%
\right) ,\left(
\begin{array}{c}
0 \\
-1 \\
0 \\
-3%
\end{array}%
\right) ,\left(
\begin{array}{c}
0 \\
0 \\
-1 \\
-3%
\end{array}%
\right) \right\} .  \label{hypo}
\end{equation}%
The minimum of $\left\Vert \cdot \right\Vert ^{2}$ on $H(\overline{a},%
\overline{b}),$ whose linear representation
\begin{equation*}
\left\{
\begin{array}{l}
x_{1}+x_{2}-x_{3}\geq -1 \\
3x_{1}+3x_{2}+3x_{3}-4x_{4}\geq 9 \\
-x_{1}-x_{2}-x_{3}\geq 1 \\
-3x_{1}+x_{2}+x_{3}\geq -1 \\
x_{1}-3x_{2}+x_{3}\geq -1 \\
-x_{1}-x_{2}+3x_{3}\geq -3%
\end{array}%
\right\}
\end{equation*}%
is obtained from \eqref{epi} and \eqref{hypo} by Fourier-Motzkin
elimination, is attained at $\left( -\frac{1}{3},-\frac{1}{3},-\frac{1}{3}%
,-3\right) .$ So,
\begin{equation*}
\rho (\mathcal{V})=\left\Vert \left( -\frac{1}{3},-\frac{1}{3},-\frac{1}{3}%
,-3\right) \right\Vert =\sqrt{\frac{28}{3}}.\newline
\end{equation*}
\end{example}

\section{Rank-1 Objective Matrix Uncertainty}

\label{sec:3}

In this section we assume that the matrix $C$ in the objective function is
uncertain and it belongs to the one-dimensional compact convex uncertainty
set in $\mathbb{R}^{m\times n}$ given by
\begin{equation*}
\mathcal{U}=\{\overline{C}+\rho uv^{\top }:\rho \in \lbrack 0,1]\},
\end{equation*}%
where $\overline{C}$ is a given $m\times n$ matrix while $u\in \mathbb{R}%
_{+}^{m}$ and $v\in \mathbb{R}^{n}$ are given vectors. This data uncertainty
set was introduced and examined in \cite[Section 3]{PLP13}.

Recall that the normal cone of a closed convex set $X$ at $\overline{x}\in X$
is
\begin{equation*}
N(X,\overline{x}):=\{u\in \mathbb{R}^{n}:u^{\top }(x-\overline{x})\leq 0,\
\forall x\in X\}.
\end{equation*}%
Moreover, the simplex $\Delta _{m}$ is defined as $\Delta _{m}:=\{\lambda
\in \mathbb{R}_{+}^{m}:\ \sum_{i=1}^{m}\lambda _{i}=1\}.$ Recall that given $%
x,y\in \mathbb{R}^{m}$, we write $x\leqq y$ ($x<y$) when $x_{i}\leq y_{i}$ ($%
x_{i}<y_{i}$, respectively) for all $i\in I:=\{1,\ldots ,m\}$. Moreover, we
write $x\leq y$ when $x\leqq y$ and $x\neq y$.

Robust efficiency means in \cite{Sitarz08} and \cite[Section 4]{PLP13},
where the constraints are deterministic, the preservation of the
corresponding property for all $C\in \mathcal{U}.$ So, this concept is very
restrictive unless the uncertainty set $\mathcal{U}$ is small in some sense
(e.g. segments emanating from $\overline{C}$). In our general framework of
uncertain objectives and constraints the following definition of robust weak
efficiency is referred to the set $X$ of robust feasible solutions.

\begin{definition}[Robust weakly efficient solution]
We say that $\overline{x}\in \mathbb{R}^{n}$ is a \emph{robust weakly
efficient solution} of $(P_{C})$ if there is no $x\in X$ such that $Cx<C%
\overline{x}$, for all $C\in \mathcal{U}$.
\end{definition}

The next characterization of the robust weakly efficient solutions in terms\
of multipliers involves the so-called\ \emph{characteristic cone} (\cite[p.
81]{GL98}) of the constraint system of $(P_{C}),$ defined as%
\begin{equation*}
C\left( \mathcal{V}\right) :={\rm cone}\left( \bigcup_{j=1}^{p}\mathcal{V}%
_{j}\right) +\mathbb{R}_{+}\left\{ \left( 0_{n},-1\right) \right\} .
\end{equation*}

If $\mathcal{V}_{j}$ is a polytope for all $\,j=1,\ldots ,p,$ then $C\left(
\mathcal{V}\right) $ is generated by the extreme points of the sets $%
\mathcal{V}_{j},\,j=1,\ldots ,p,$ together with the vector $\left(
0_{n},-1\right) .$ So, $C\left( \mathcal{V}\right) $ is a polyhedral convex
cone.

If $\mathcal{V}_{j}$ is a compact convex set for all $\,j=1,\ldots ,p$ and
the strict robust feasibility condition
\begin{equation}
\{x\in \mathbb{R}^{n}:a_{j}^{\top }x>b_{j},\,\forall (a_{j},b_{j})\in
\mathcal{V}_{j},\,j=1,\ldots ,p\}\neq \emptyset  \label{SlaterLinear}
\end{equation}%
holds, then, according to \cite[Theorem 5.3 (ii)]{GL98}, ${\rm cone}\left(
\bigcup_{j=1}^{p}\mathcal{V}_{j}\right) $ is closed, an this in turn implies
that $C\left( \mathcal{V}\right) $ is closed too.

\begin{theorem}[Robust weakly efficient solutions]
\label{th:1} The point $\overline{x}\in X$ is a robust weakly solution of $%
(P_{C})$ if and only if there exist $\lambda ,\tilde{\lambda}\in \Delta _{m}$
such that
\begin{equation*}
\overline{C}^{\top }\lambda \in -N(X,\overline{x})\mbox{ and }(\overline{C}%
+uv^{\top })^{\top }\tilde{\lambda}\in -N(X,\overline{x}).
\end{equation*}%
Moreover, if $\mathcal{V}_{j}$ is convex, $\,j=1,\ldots ,p,$ and $C\left(
\mathcal{V}\right) $ is closed, then the robust weak efficiency of $%
\overline{x}\in X$ is further equivalent to the condition that there exist $%
\lambda ,\tilde{\lambda}\in \Delta _{m}$ and $(a_{j},b_{j}),(\tilde{a}_{j},%
\tilde{b}_{j})\in \mathcal{V}_{j}$, $\mu _{j},\tilde{\mu}_{j}\geq 0$, $%
j=1,\ldots ,p$, such that
\begin{equation*}
\overline{C}^{\top }\lambda =\sum_{j=1}^{p}\mu _{j}a_{j} \quad \mbox{and}
\quad \mu _{j}(a_{j}^{\top }\overline{x}-b_{j})=0, j=1,\ldots ,p,
\end{equation*}
and
\begin{equation*}
(\overline{C}+uv^{\top })^{\top }\tilde{\lambda}=\sum_{j=1}^{p}\tilde{\mu}%
_{j}a_{j} \quad \mbox{and} \quad \tilde{\mu}_{j}(\tilde{a}_{j}^{\top }%
\overline{x}-\tilde{b}_{j})=0, j=1,\ldots ,p.
\end{equation*}
\end{theorem}

\begin{proof}
Let $\overline{x}\in X$ be a robust weakly efficient solution. Then, we have
for each $C\in \mathcal{U}$, there exist no $x\in X$ such that $Cx<C%
\overline{x}.$ By \cite[Prop. 18 (iii)]{GGT13}, this is equivalent to the
fact that
\begin{equation*}
(\forall C\in \mathcal{U}),(\exists \lambda \in \mathbb{R}_{+}^{m}\diagdown
\left\{ 0_{m}\right\} )(C^{\top }\lambda \in -N(X,\overline{x})).
\end{equation*}%
As $N(X,\overline{x})$ is a cone, by normalization, we may assume that $%
\lambda \in \Delta _{m}$, and so, $\overline{x}$ is a robust weakly
efficient solution if and only if
\begin{equation}
(\forall C\in \mathcal{U}),(\exists \lambda \in \Delta _{m})(C^{\top
}\lambda \in -N(X,\overline{x})).  \label{eq:popo}
\end{equation}%
To see the first assertion, it suffices to show that \eqref{eq:popo} is
further equivalent to
\begin{equation}
(\exists \,\lambda ,\tilde{\lambda}\in \Delta _{m})(\overline{C}^{\top
}\lambda \in -N(X,\overline{x})\mbox{ and }(\overline{C}+uv^{\top })^{\top }%
\tilde{\lambda}\in -N(X,\overline{x})).  \label{popo1}
\end{equation}%
To see the equivalence, we only need to show that \eqref{popo1} implies %
\eqref{eq:popo} when $u\neq 0_{m}$ (otherwise $\mathcal{U}$\ is a singleton
set). To achieve this, suppose that \eqref{popo1} holds and fix an arbitrary
$C\in \mathcal{U}$. Then there exists $\alpha \in \lbrack 0,1]$ such that $C=%
\overline{C}+\alpha uv^{\top }.$\newline

Define $\tau :=\frac{\left( 1-\alpha \right) \tilde{\lambda}^{\top }u}{%
\left( 1-\alpha \right) \tilde{\lambda}^{\top }u+\alpha \lambda ^{\top }u}$
and $\gamma :=\tau \lambda +\left( 1-\tau \right) \tilde{\lambda}\geq 0_{m}$%
. As $\lambda ,\tilde{\lambda}\in \Delta _{m}$ and $u\in \mathbb{R}_{+}^{m}$%
, we see that $\tau \in \lbrack 0,1]$ and $\gamma \in \Delta _{m}$.
Moreover, we have
\begin{eqnarray}
&&\tau \alpha (uv^{\top })^{\top }\lambda -(1-\alpha )(1-\tau )(uv^{\top
})^{\top }\tilde{\lambda}\medskip  \notag  \label{eq:pre} \\
&=&\tfrac{\left( 1-\alpha \right) \tilde{\lambda}^{\top }u}{\left( 1-\alpha
\right) \tilde{\lambda}^{\top }u+\alpha \lambda ^{\top }u}\alpha (uv^{\top
})^{\top }\lambda -\tfrac{\alpha \lambda ^{\top }u}{\left( 1-\alpha \right)
\tilde{\lambda}^{\top }u+\alpha \lambda ^{\top }u}(1-\alpha )(uv^{\top
})^{\top }\tilde{\lambda}\medskip  \notag \\
&=&\tfrac{\left( 1-\alpha \right) \tilde{\lambda}^{\top }u}{\left( 1-\alpha
\right) \tilde{\lambda}^{\top }u+\alpha \lambda ^{\top }u}\alpha (u^{\top
}\lambda )v-\tfrac{\alpha \lambda ^{\top }u}{\left( 1-\alpha \right) \tilde{%
\lambda}^{\top }u+\alpha \lambda ^{\top }u}(1-\alpha )(u^{\top }\tilde{%
\lambda})v=0_{m}.
\end{eqnarray}%
Now,

\begin{eqnarray*}
C^{\top }\gamma &=&\left( \overline{C}+\alpha uv^{\top }\right) ^{\top
}(\tau \lambda +\left( 1-\tau \right) \tilde{\lambda})\medskip \\
&=&\tau \overline{C}^{\top }\lambda +\tau \alpha (uv^{\top })^{\top }\lambda
+(1-\tau )\left( \overline{C}+\alpha uv^{\top }\right) ^{\top }\tilde{\lambda%
}\medskip \\
&=&\tau \overline{C}^{\top }\lambda +\tau \alpha (uv^{\top })^{\top }\lambda
+(1-\tau )(\overline{C}+uv^{\top })^{\top }\tilde{\lambda}-(1-\alpha
)(1-\tau )(uv^{\top })^{\top }\tilde{\lambda}\medskip \\
&=&\tau \overline{C}^{\top }\lambda +(1-\tau )(\overline{C}+uv^{\top
})^{\top }\tilde{\lambda}\in N(X,\overline{x}).
\end{eqnarray*}%
where the fourth equality follows from \eqref{eq:pre} and the last relation
follows from \eqref{popo1} and the convexity of $N(X,\overline{x})$.

To see the second assertion, we assume that $\mathcal{V}_{j}$ is convex, $%
\,j=1,\ldots ,p,$ and $C\left( \mathcal{V}\right) $ is closed. We only need
to show
\begin{equation*}
N(X,\overline{x})=\left\{ -\sum_{j=1}^{p}\mu _{j}a_{j}:(a_{j},b_{j})\in
\mathcal{V}_{j},\mu _{j}\geq 0\mbox{ and }\mu _{j}(a_{j}^{\top }\overline{x}%
-b_{j})=0,j=1,\ldots ,p\right\} .
\end{equation*}

The system $\left\{ a^{\top }x\geq b,\left( a,b\right) \in T\right\} ,$ with
$T=\left( \bigcup_{j=1}^{p}\mathcal{V}_{j}\right) ,$ is a linear
representation of $X.$ Thus, $u\in N(X,\overline{x})$ if and only if the
inequality $-u^{\top }x\geq $ $-u^{\top }\overline{x}$ is consequence of $%
\left\{ a^{\top }x\geq b,\left( a,b\right) \in T\right\} $ if and only if
(by the Farkas Lemma, \cite[Corollary 3.1.2]{GL98})
\begin{equation*}
-\left( u,u^{\top }\overline{x}\right) \in {\rm cone}T+\mathbb{R}%
_{+}\left\{ \left( 0_{n},-1\right) \right\} .
\end{equation*}%
This is equivalent to assert the existence of a finite subset $S$ of $T,$
corresponding non-negative scalars $\lambda _{s},$ $s\in S,$ and $\mu \geq
0, $ such that
\begin{equation}
-\left( u,u^{\top }\overline{x}\right) =\sum\nolimits_{\left( a,b\right) \in
S}\lambda _{\left( a,b\right) }\left( a,b\right) +\mu \left( 0_{n},-1\right)
.  \label{mu}
\end{equation}%
Multiplying by $\left( \overline{x},-1\right) $ both members of \eqref{mu}
we get $\mu =0,$ so that \eqref{mu} is equivalent to
\begin{equation}
-u=\sum\nolimits_{\left( a,b\right) \in S}\lambda _{\left( a,b\right) }a%
\text{ and }\lambda _{\left( a,b\right) }(a^{\top }\overline{x}-b)=0,\left(
a,b\right) \in S.  \label{mu1}
\end{equation}

Finally, since $S\subset \bigcup_{j=1}^{p}\mathcal{V}_{j},$ we can write $%
S=\bigcup_{j=1}^{p} S_{j},$ with $S_{j}\subset \mathcal{V}_{j},$ $j=1,\ldots
,p,$ and $S_{i}\cap S_{j}=\emptyset $ when $i\neq j.$ Let $\mu
_{j}:=\sum\nolimits_{\left( a,b\right) \in S_{j}}\lambda _{\left( a,b\right)
},$ $j=1,\ldots ,p.$ If $\mu _{j}\neq 0$ one has, by convexity of $\mathcal{V%
}_{j},$
\begin{equation*}
(a_{j},b_{j}):=\tfrac{\sum\nolimits_{\left( a,b\right) \in S_{j}}\lambda
_{\left( a,b\right) }\left( a,b\right) }{\mu _{j}}\in \mathcal{V}_{j}.
\end{equation*}%
Take $(a_{j},b_{j})\in \mathcal{V}_{j}$ arbitrarily when $\mu _{j}=0.$\ Then
we get from \eqref{mu1} that%
\begin{equation*}
-u=\sum_{j=1}^{p}\mu _{j}a_{j}\text{ and }\mu _{j}(a_{j}^{\top }\overline{x}%
-b_{j})=0,j=1,\ldots ,p.
\end{equation*}

Thus, the conclusion follows.
\end{proof}

\smallskip

In the definition of the rank-1 objective data uncertainty set, $\mathcal{U}%
=\{\overline{C}+\rho uv^{\top }:\rho \in \lbrack 0,1]\}$, we require that $%
u\in \mathbb{R}_{+}^{m}$. The following example (inspired in \cite[Example
3.3]{PLP13}) illustrates that if this non-negativity requirement is dropped,
then the above solution characterization in Theorem \ref{th:1} may fail.

\begin{example}[Non-negativity requirement for rank-1 objective data
uncertainty]
Let
\begin{equation*}
\overline{C}=\left(
\begin{array}{ccc}
-3 & -1 & -2 \\
0 & -1 & -2%
\end{array}%
\right) ,\text{ }u=\left(
\begin{array}{c}
-1 \\
1%
\end{array}%
\right) \notin \mathbb{R}_{+}^{2}\text{ }\mbox{ and }\ v=\left(
\begin{array}{c}
0 \\
-3 \\
0%
\end{array}%
\right) .
\end{equation*}%
Consider the uncertain multiobjective optimization problem
\begin{equation}
\quad \text{V-}\min \left\{ Cx:a_{j}^{\top }x\geq b_{j},\;\forall
(a_{j},b_{j})\in \mathcal{V}_{j},\,j=1,\ldots ,4\right\} ,  \label{problem}
\end{equation}%
where the objective data matrix $C$ is an element of
\begin{equation*}
\{\overline{C}+\rho uv^{\top }:\rho \in \lbrack 0,1]\}=\left\{ \left(
\begin{array}{ccc}
-3 & -1 & -2 \\
0 & -1 & -2%
\end{array}%
\right) +\rho \left(
\begin{array}{ccc}
0 & 3 & 0 \\
0 & -3 & 0%
\end{array}%
\right) :\rho \in \lbrack 0,1]\right\}
\end{equation*}%
and the uncertainty sets for the constraints are the convex polytopes
\begin{equation*}
\mathcal{V}_{1}={\rm conv}\left\{ \left(
\begin{array}{c}
-2 \\
-1 \\
-2 \\
-6%
\end{array}%
\right) ,\left(
\begin{array}{c}
-1 \\
-2 \\
-2 \\
-6%
\end{array}%
\right) \right\} \text{ and }\mathcal{V}_{2}={\rm conv}\left\{ \left(
\begin{array}{c}
-1 \\
0 \\
0 \\
-3%
\end{array}%
\right) ,\left(
\begin{array}{c}
0 \\
-1 \\
0 \\
-3%
\end{array}%
\right) ,\left(
\begin{array}{c}
0 \\
0 \\
-1 \\
-3%
\end{array}%
\right) \right\} .
\end{equation*}%
Note that the robust feasible set is%
\begin{equation*}
X=\{x\in \mathbb{R}^{n}:a_{j}^{\top }x\geq b_{j},\,\forall (a_{j},b_{j})\in
\mathcal{V}_{j},j=1,2\}=\left\{ \overline{a}_{j}^{\top }x\geq \overline{b}%
_{j},j=1,...,5\right\} ,
\end{equation*}%
where $\left\{ \overline{a}_{j}^{\top }x\geq \overline{b}_{j},j=1,...,5%
\right\} $ is the set in \eqref{hypo}. It can be checked that $\overline{x}%
=(1,1,3/2)\in X$ and so,
\begin{equation*}
N(X,\overline{x})=\left\{ \mu _{1}\left(
\begin{array}{c}
2 \\
1 \\
2%
\end{array}%
\right) +\mu _{2}\left(
\begin{array}{c}
1 \\
2 \\
2%
\end{array}%
\right) :\mu _{1}\geq 0,\mu _{2}\geq 0\right\} .
\end{equation*}%
Let $\lambda =(2/3,1/3)^{\top }$ and $\tilde{\lambda}=(1/3,2/3)^{\top }$.
Then, we have
\begin{equation*}
\overline{C}^{\top }\lambda \in -N(X,\overline{x})\mbox{ and }(\overline{C}%
+uv^{\top })^{\top }\tilde{\lambda}\in -N(X,\overline{x}).
\end{equation*}%
On the other hand, for
\begin{equation*}
C=\left(
\begin{array}{ccc}
-3 & -1 & -2 \\
0 & -1 & -2%
\end{array}%
\right) +\frac{1}{2}\left(
\begin{array}{ccc}
0 & 3 & 0 \\
0 & -3 & 0%
\end{array}%
\right) =\left(
\begin{array}{ccc}
-3 & \frac{1}{2} & -2 \\
0 & -\frac{5}{2} & -2%
\end{array}%
\right) \in \mathcal{U},
\end{equation*}%
and $x=(0,0,3)^{\top }\in X$, we see that
\begin{equation*}
Cx=\left(
\begin{array}{c}
-6 \\
-6%
\end{array}%
\right) <\left(
\begin{array}{c}
-\frac{11}{2}\medskip \\
-\frac{11}{2}%
\end{array}%
\right) =C\overline{x}.
\end{equation*}%
So, $\overline{x}$ is not a weakly efficient solution of \eqref{problem}.
Thus, the above solution characterization fails.
\end{example}

In the case where the constraints are uncertainty free, i.e. the sets $%
\mathcal{V}_{j}$ are all singletons, we obtain the following solution
characterization for robust multiobjective optimization problem with
rank-one objective uncertainty.

\begin{corollary}
\label{CorolRank1}Let $\mathcal{V}_{j}=\left\{ (\overline{a}_{j},\overline{b}%
_{j})\right\} ,$ $j=1,\ldots ,p,$ and $\overline{x}\in X.$ Then, the
following statements are equivalent:\newline
(i) $\overline{x}$ is a robust weakly efficient solution;\newline
(ii) there exist $\lambda ,\tilde{\lambda}\in \Delta _{m}$ such that
\begin{equation*}
\overline{C}^{\top }\lambda \in -N(X,\overline{x})\mbox{ and }(\overline{C}%
+uv^{\top })^{\top }\tilde{\lambda}\in -N(X,\overline{x});
\end{equation*}%
\newline
(iii) there exist $\lambda ,\tilde{\lambda}\in \Delta _{m}$ and $\mu _{j},%
\tilde{\mu}_{j}\geq 0$, $j=1,\ldots ,p$, such that%
\begin{equation*}
\overline{C}^{\top }\lambda =\sum_{j=1}^{p}\mu _{j}\overline{a}_{j}%
\mbox{
and }\mu _{j}(\overline{a}_{j}^{\top }\overline{x}-\overline{b}%
_{j})=0,j=1,\ldots ,p,
\end{equation*}%
and%
\begin{equation*}
(\overline{C}+uv^{\top })^{\top }\tilde{\lambda}=\sum_{j=1}^{p}\tilde{\mu}%
_{j}\overline{a}_{j}\mbox{ and }\tilde{\mu}_{j}(\overline{a}_{j}^{\top }%
\overline{x}-\overline{b}_{j})=0,j=1,\ldots ,p;
\end{equation*}%
\newline
(iv) $\overline{x}$ is a weakly efficient solution for the problems
\begin{equation*}
\begin{array}{lrl}
(P_{0}) & \text{V-}\min & \overline{C}x \\
& \text{s.t.} & \overline{a}_{j}^{\top }x\geq \overline{b}_{j},\,j=1,\ldots
,p,%
\end{array}%
\end{equation*}%
and
\begin{equation*}
\begin{array}{lrl}
(P_{1}) & \text{V-}\min & (\overline{C}+uv^{\top })x \\
& \text{s.t.} & \overline{a}_{j}^{\top }x\geq \overline{b}_{j},\,j=1,\ldots
,p.%
\end{array}%
\end{equation*}
\end{corollary}

\begin{proof}
Let $\mathcal{V}_{j}=\{\left( \overline{a}_{j},\overline{b}_{j}\right) \}$, $%
j=1,\ldots ,p$. The equivalences (i)$\Leftrightarrow $(ii)$\Leftrightarrow $%
(iii) come from Theorem \ref{th:1}, taking into account that all the
uncertainty sets $\mathcal{V}_{j}$ are polytopes. Note that (i)$\Rightarrow $%
(iv) always holds. Finally, the implication (iv)$\Rightarrow $(ii) is
immediate by the usual characterization for weakly efficient solutions (e.g.
see \cite[Prop. 18(iii)]{GGT13}). Thus, the conclusion follows.
\end{proof}

\begin{remark}
The equivalence (i)$\Leftrightarrow $(iii)\ in Corollary \ref{CorolRank1},
on robust weakly efficient solutions of uncertain vector linear programming
problems, can be seen as a counterpart of \cite[Theorem 3.1]{PLP13}, on
robust efficient solutions of the same type of problems.
\end{remark}

\section{Tractable Classes of Robust Multi-Objective LPs}

\label{sec:4}

In this Section, we provide various classes of commonly used uncertainty
sets determining the robust feasible set
\begin{equation*}
X=\{x\in \mathbb{R}^{n}:a_{j}^{\top }x\geq b_{j},\,\forall (a_{j},b_{j})\in
\mathcal{V}_{j},\,j=1,\ldots ,p\},
\end{equation*}%
under which one can numerically check whether a robust feasible point is a
robust weakly efficient solution or not. Throughout this Section we assume
that the objective function of $(P_{C})$ satisfies the rank-1 matrix data
uncertainty, as defined in Section \ref{sec:3}. We begin with the simple box
constraint data uncertainty.

\subsection{Box constraint data Uncertainty}

Consider
\begin{equation}
\mathcal{V}_{j}=[\underline{a}_{j},\overline{a}_{j}]\times \lbrack
\underline{b}_{j},\overline{b}_{j}],  \label{VBox}
\end{equation}%
where $\underline{a}_{j},\overline{a}_{j}\in \mathbb{R}^{n}$ and $\underline{%
b}_{j},\overline{b}_{j}\in \mathbb{R},$ $j=1,\ldots ,p$. Denote the extreme
points of $[\underline{a}_{j},\overline{a}_{j}]$ by $\{\hat{a}%
_{j}^{(1)},\ldots ,\hat{a}_{j}^{(2^{n})}\}$.

\begin{theorem}
\label{th:3} Let $\mathcal{V}_{j}$ be as in \eqref{VBox}, $\,j=1,\ldots ,p$.
The point $\overline{x}\in X$ is a robust weakly efficient solution of $%
(P_{C})$ if and only if there exist $\lambda ,\tilde{\lambda}\in \Delta _{m}$
and $\mu _{j}^{(l)},\tilde{\mu}_{j}^{(l)}\geq 0$ such that
\begin{equation*}
\overline{C}^{\top }\lambda =\sum_{j=1}^{p}\sum_{l=1}^{2^{n}}\mu _{j}^{(l)}%
\hat{a}_{j}^{(l)}\quad \mbox{and}\quad \mu _{j}^{(l)}\bigg((\hat{a}%
_{j}^{(l)})^{\top }\overline{x}-\overline{b}_{j}\bigg)=0,\,j=1,\ldots
,p,l=1,\ldots ,2^{n},
\end{equation*}%
and
\begin{equation*}
(\overline{C}+uv^{\top })^{\top }\tilde{\lambda}=\sum_{j=1}^{p}%
\sum_{l=1}^{2^{n}}\tilde{\mu}_{j}^{(l)}\hat{a}_{j}^{(l)}\quad \mbox{and}%
\quad \tilde{\mu}_{j}^{(l)}\bigg((\hat{a}_{j}^{(l)})^{\top }\overline{x}-%
\overline{b}_{j}\bigg)=0,\,j=1,\ldots ,p,l=1,\ldots ,2^{n}.
\end{equation*}
\end{theorem}

\begin{proof}
Let $\overline{x}$ be a robust weakly efficient solution of $(P_{C})$. Note
that $X$ can be rewritten as
\begin{eqnarray*}
X &=&\left\{ x\in \mathbb{R}^{n}:a_{j}^{\top }x-b_{j}\geq 0\mbox{ for all }%
(a_{j},b_{j})\in \lbrack \underline{a}_{j},\overline{a}_{j}]\times \lbrack
\underline{b}_{j},\overline{b}_{j}]\right\} \\
&=&\left\{ x\in \mathbb{R}^{n}:(a_{j}^{(l)})^{\top }x-\overline{b}_{j}\geq
0,l=1,\ldots ,2^{n},j=1,\ldots ,p\right\} .
\end{eqnarray*}%
Then, we have
\begin{equation*}
N(X,\overline{x})=\left\{ -\sum_{j=1}^{p}\sum_{l=1}^{2^{n}}\mu _{j}^{(l)}%
\hat{a}_{j}^{(l)}:\mu _{j}^{(l)}\bigg((\hat{a}_{j}^{(l)})^{\top }\overline{x}%
-\overline{b}_{j}\bigg)=0,\mu _{j}^{(l)}\geq 0,\forall l,\forall j\right\} .
\end{equation*}%
Since $\mathcal{V}_{j}$ is a convex polytope for $\,j=1,\ldots ,p,$ the
conclusion follows from Theorem \ref{th:1}.
\end{proof}

\smallskip

It is worth noting, from Theorem \ref{th:3}, that one can determine whether
or not a given robust feasible point $\overline{x}$ of $(P_{C})$ under the
box constraint data uncertainty is a robust weakly efficient solution by
solving finitely many linear equalities.

\subsection{Norm constraint data uncertainty}

Consider the constraint data uncertainty set
\begin{equation}
\mathcal{V}_{j}=\left\{ \overline{a}_{j}+\delta _{j}\overline{v}_{j}:%
\overline{v}_{j}\in \mathbb{R}^{n},\Vert Z_{j}\overline{v}_{j}\Vert _{s}
\leq 1\right\} \times \lbrack \underline{b}_{j},\overline{b}_{j}],
\label{VNorm}
\end{equation}
where $\overline{a}_{j}\in \mathbb{R}^{n}$, $\overline{b}_{j}\in \mathbb{R}$%
, $Z_{j}$ is an invertible symmetric $(n\times n)$ matrix, $j=1,\ldots ,p,$
and let $\Vert \cdot \Vert _{s}$ denote the $s$-norm, $s\in \lbrack
1,+\infty ]$, defined by
\begin{equation*}
\Vert x\Vert _{s}=\left\{
\begin{array}{lll}
\sqrt[s]{\sum_{i=1}^{n}|x_{i}|^{s}} & \mbox{ if } & s\in \lbrack 1,+\infty
),\medskip \\
\max \{|x_{i}|:1\leq i\leq n\} & \mbox{ if } & s=+\infty .%
\end{array}%
\right.
\end{equation*}%
Moreover, we define $s^{\ast }\in \lbrack 1,+\infty ]$ to be the number so
that $\frac{1}{s}+\frac{1}{s^{\ast }}=1$. The following simple facts about $%
s $-norms will be used later on. First, the dual norm of the $s$-norm is the
$s^{\ast }$-norm, that is,
\begin{equation*}
\sup_{\Vert x\Vert _{s}\leq 1}u^{\top }x=\Vert u\Vert _{s^{\ast }}%
\mbox{ for
all }u\in \mathbb{R}^{n}.
\end{equation*}%
Second, $\partial (\Vert \cdot \Vert _{s^{\ast }})(u)=\{v:\Vert v\Vert
_{s}\leq 1,v^{\top }u=\Vert u\Vert _{s^{\ast }}\}$ where $\partial f(x)$
denotes the usual convex subdifferential of a convex function $f:\mathbb{R}%
^{n}\rightarrow \mathbb{R}$ at $x\in \mathbb{R}^{n},$ i.e.
\begin{equation*}
\partial f(x)=%
\begin{array}{cc}
\{z\in \mathbb{R}^{n}:z^{\top }(y-x)\leq f(y)-f(x)\,\forall \,y\in \mathbb{R}%
^{n}\}. &
\end{array}%
\end{equation*}%
In this case, we have the following characterization of robust weakly
efficient solutions.

\begin{theorem}
\label{NA} Let $\mathcal{V}_{j}$ be as in \eqref{VNorm}, $\,j=1,\ldots ,p,$
and suppose that there exists $x_{0}\in \mathbb{R}^{n}$ such that
\begin{equation}
\overline{a}_{j}^{\top }x_{0}-\overline{b}_{j}-\delta \Vert
Z_{j}^{-1}x_{0}\Vert _{s^{\ast }}>0,j=1,\ldots ,p.  \label{SlaterNorm}
\end{equation}%
Then, a point $\overline{x}\in X$ is a robust weakly efficient solution of $%
(P_{C})$ if and only if there exist $\lambda ,\tilde{\lambda}\in \Delta _{m}$%
, $\mu ,\tilde{\mu}\in \mathbb{R}_{+}^{p}$ and $w_{j},\tilde{w}_{j}\in
\mathbb{R}^{n},$ with $\Vert w_{j}\Vert _{s}\leq 1$ and $\Vert \tilde{w}%
_{j}\Vert _{s}\leq 1,$ such that
\begin{equation*}
-\lambda ^{\top }\overline{C}\overline{x}=\sum_{j=1}^{p}\mu _{j}\overline{b}%
_{j}\quad \mbox{and}\quad \overline{C}^{\top }\lambda +\sum_{j=1}^{p}\mu
_{j}(\overline{a}_{j}-\delta Z_{j}^{-1}w_{j})=0_{n}.
\end{equation*}%
and
\begin{equation*}
-\tilde{\lambda}^{\top }(\overline{C}+uv^{\top })\overline{x}=\sum_{j=1}^{p}%
\tilde{\mu}_{j}\overline{b}_{j}\quad \mbox{and}\quad (\overline{C}+uv^{\top
})^{\top }\tilde{\lambda}+\sum_{j=1}^{p}\tilde{\mu}_{j}(\overline{a}%
_{j}-\delta Z_{j}^{-1}\tilde{w}_{j})=0_{n}.
\end{equation*}
\end{theorem}

\begin{proof}
Note that $X$ can be rewritten as
\begin{eqnarray*}
X &=&\left\{ x\in \mathbb{R}^{n}:\overline{a}_{j}^{\top }x-b_{j}+\delta (%
\overline{v}_{j})^{\top }x\geq 0\mbox{ for all }\Vert Z_{j}\overline{v}%
_{j}\Vert _{s}\leq 1,b_{j}\in \lbrack \underline{b}_{j},\overline{b}%
_{j}],j=1,\ldots ,p\right\} \\
&=&\left\{ x\in \mathbb{R}^{n}:\overline{a}_{j}^{\top }x-b_{j}+\delta
(Z_{j}^{-1}\overline{u}_{j})^{\top }x\geq 0\mbox{ for all }\Vert \overline{u}%
_{j}\Vert _{s}\leq 1,b_{j}\in \lbrack \underline{b}_{j},\overline{b}%
_{j}],j=1,\ldots ,p\right\} \\
&=&\left\{ x\in \mathbb{R}^{n}:\overline{a}_{j}^{\top }x-\overline{b}%
_{j}-\delta \Vert Z_{j}^{-1}x\Vert _{s^{\ast }}\geq 0,j=1,\ldots ,p\right\} .
\end{eqnarray*}%
Since $\mathcal{V}_{j}$ is a compact convex set for $\,j=1,\ldots ,p$\ and
the strict robust feasibility\ condition \eqref{SlaterLinear} holds as a
consequence of \eqref{SlaterNorm}, the conclusion will follow from Theorem %
\ref{th:1} if we show that
\begin{equation*}
N(X,\overline{x})=\left\{ u:\exists \mu _{j}\geq 0,\Vert w_{j}\Vert
_{s^{\ast }}\leq 1\mbox { s.t. }-u^{\top }\overline{x}=\sum_{j=1}^{p}\mu _{j}%
\overline{b}_{j}\mbox{ and }u+\sum_{j=1}^{p}\mu _{j}(\overline{a}_{j}-\delta
Z_{j}^{-1}w_{j})=0_{n}\right\} .
\end{equation*}%
To see this, let $u\in N(X,\overline{x})$. Then, $\overline{x}$ is a
solution of the following convex optimization problem:
\begin{equation*}
\min \left\{ -u^{\top }x:\overline{a}_{j}^{\top }x-\overline{b}_{j}-\delta
\Vert Z_{j}^{-1}x\Vert _{s^{\ast }}\geq 0,j=1,\ldots ,p\right\}
\end{equation*}%
As the strict feasibility condition \eqref{SlaterNorm} holds, by the
Lagrangian duality, there exist $\mu _{j}\geq 0,$ $j=1,\ldots ,p,$ such that
\begin{equation*}
-u^{\top }\overline{x}=\min_{x\in \mathbb{R}^{n}}\left\{ (-u^{\top
}x)+\sum_{j=1}^{p}\mu _{j}\left( -\overline{a}_{j}^{\top }x+\overline{b}%
_{j}+\delta \Vert Z_{j}^{-1}x\Vert _{s^{\ast }}\right) \right\} .
\end{equation*}%
As $\overline{x}\in X$, this implies that $\mu _{j}\left( -\overline{a}%
_{j}^{\top }\overline{x}+\overline{b}_{j}+\delta \Vert Z_{j}^{-1}\overline{x}%
\Vert _{s^{\ast }}\right) =0,j=1,\ldots ,p$, and so, the function $%
h(x):=(-u^{\top }x)+\sum_{j=1}^{p}\mu _{j}\left( -\overline{a}_{j}^{\top }x+%
\overline{b}_{j}+\delta \Vert Z_{j}^{-1}x\Vert _{s^{\ast }}\right) $ attains
its minimum on $X$ at $\overline{x}$ and $\min_{x\in \mathbb{R}%
^{n}}h(x)=-u^{\top }\overline{x}$. This implies that $0_{n}\in \partial h(%
\overline{x})$, and so, there exist $w_{j}\in \mathbb{R}^{n}$ with $\Vert
w_{j}\Vert _{s}\leq 1$ such that $w_{j}^{\top }(Z_{j}^{-1}\overline{x}%
)=\Vert Z_{j}^{-1}\overline{x}\Vert _{s^{\ast }}$ and
\begin{equation*}
u+\sum_{j=1}^{p}\mu _{j}\left( \overline{a}_{j}-\delta
Z_{j}^{-1}w_{j}\right) =0_{n}.
\end{equation*}%
This together with $h(\overline{x})=-u^{\top }\overline{x}$ gives us that $%
-u^{\top }\overline{x}=\sum_{j=1}^{p}\mu _{j}\overline{b}_{j}$. Then, we
have
\begin{equation*}
N(X,\overline{x})\subset \left\{ u:\exists \mu _{j}\geq 0,\Vert w_{j}\Vert
_{s^{\ast }}\leq 1\mbox {
s.t. }-u^{\top }\overline{x}=\sum_{j=1}^{p}\mu _{j}\overline{b}_{j}%
\mbox{
and }u+\sum_{j=1}^{p}\mu _{j}\left( \overline{a}_{j}-\delta
Z_{j}^{-1}w_{j}\right) =0_{n}\right\} .
\end{equation*}%
To see the reverse inclusion, let $u\in \mathbb{R}^{n}$ with $-u^{\top }%
\overline{x}=\sum_{j=1}^{p}\mu _{j}\overline{b}_{j}\mbox{ and }%
u+\sum_{j=1}^{p}\mu _{j}\left( \overline{a}_{j}-\delta
Z_{j}^{-1}w_{j}\right) =0_{n}$ for some $\mu _{j}\geq 0,\Vert w_{j}\Vert
_{s^{\ast }}\leq 1$. Then, for all $x\in X$,
\begin{eqnarray*}
u^{\top }(x-\overline{x})=u^{\top }x+\sum_{j=1}^{p}\mu _{j}\overline{b}_{j}
&=&-\left( \sum_{j=1}^{p}\mu _{j}(\overline{a}_{j}-\delta
Z_{j}^{-1}w_{j})\right) ^{\top }x+\sum_{j=1}^{p}\mu _{j}\overline{b}_{j} \\
&=&\sum_{j=1}^{p}\mu _{j}\big(-\overline{a}_{j}^{\top }x+\overline{b}%
_{j}+\delta \left( Z_{j}^{-1}w_{j}\right) ^{\top }x\big) \\
&\leq &\sum_{j=1}^{p}\mu _{j}\big(-\overline{a}_{j}^{\top }x+\overline{b}%
_{j}+\delta \Vert Z_{j}^{-1}x\Vert _{s^{\ast }}\big)\leq 0,
\end{eqnarray*}%
where the inequality follows from $\Vert w_{j}\Vert _{s^{\ast }}=\max_{\Vert
u\Vert _{s}\leq 1}w_{j}^{\top }u$ (and hence, $w_{j}^{\top }v\leq \Vert
w_{j}\Vert _{s^{\ast }}\Vert v\Vert _{s}\leq \Vert v\Vert _{s}$ for all $%
v\in \mathbb{R}^{n}$). So, $u\in N(X,\overline{x})$ and hence, the
conclusion follows.
\end{proof}

\smallskip

Theorem \ref{NA} shows that one can determine whether a robust feasible
point $\overline{x}$ under norm data uncertainty is a robust weakly
efficient solution or not by solving finitely many $s$th-order cone systems
(that is, linear equations where the variable lies in the ball determined by
the $\Vert \cdot \Vert_{s}$-norm) as long as the strict feasibility
condition \eqref{SlaterNorm} is satisfied.

\subsection{Ellipsoidal constraint data uncertainty}

In this subsection we consider the case where the constraint data are
uncertain and belong to the ellipsoidal constraint data uncertainty sets
\begin{equation}
\mathcal{V}_{j}=\{\overline{a}_{j}^{0}+\sum_{l=1}^{q_{j}}v_{j}^{l}\overline{a%
}_{j}^{l}:\Vert (v_{j}^{1},\dots ,v_{j}^{q_{j}})\Vert \leq 1\}\times \lbrack
\underline{b}_{j},\overline{b}_{j}],  \label{VEllipsoid}
\end{equation}
where $\overline{a}_{j}^{l}\in \mathbb{R}^{n}$, $l=0,1,\ldots ,q_{j}$, $%
q_{j}\in \mathbb{N}$ and $\underline{b}_{j},\overline{b}_{j}\in \mathbb{R},$
$j=1,...,p$.

\begin{theorem}
Let $\mathcal{V}_{j},$ $\,j=1,\ldots ,p,$ be as in \eqref{VEllipsoid} and
suppose that there exists $x_{0}\in \mathbb{R}^{n}$ such that%
\begin{equation}
(\overline{a}_{j}^{0})^{\top }x_{0}-\overline{b}_{j}-\Vert \big((\overline{a}%
_{j}^{1})^{\top }x_{0},\ldots ,(\overline{a}_{j}^{q_{j}})^{\top }x_{0}\big)%
\Vert >0,\,j=1,\ldots ,p.  \label{SlaterEllipsoid}
\end{equation}%
Then, a point $\overline{x}\in X$ is a robust weakly efficient solution of $%
(P_{C})$ if and only if there exist $\lambda ,\tilde{\lambda}\in \Delta _{m}$%
, $\mu ,\tilde{\mu}\in \mathbb{R}_{+}^{p}$ and $w,\tilde{w}\in \mathbb{R}^{n}
$ with $\Vert w\Vert \leq 1$ and $\Vert \tilde{w}\Vert \leq 1$ such that
\begin{equation*}
-\lambda ^{\top }\overline{C}\overline{x}=\sum_{j=1}^{p}\mu _{j}\overline{b}%
_{j}\quad \mbox{and}\quad -\overline{C}^{\top }\lambda -\sum_{j=1}^{p}\mu
_{j}(\overline{a}_{j}^{0}-y_{j})=0_{m}
\end{equation*}%
and
\begin{equation*}
-\lambda ^{\top }(\overline{C}+uv^{\top })\overline{x}=\sum_{j=1}^{p}\mu _{j}%
\overline{b}_{j}\quad \mbox{and}\quad -(\overline{C}+uv^{\top })^{\top
}\lambda -\sum_{j=1}^{p}\mu _{j}(\overline{a}_{j}^{0}-y_{j})=0_{m},
\end{equation*}%
where $y_{j}=\big((\overline{a}_{j}^{1})^{\top }w,\ldots ,(\overline{a}%
_{j}^{q_{j}})^{\top }w\big)^{\top }$.
\end{theorem}

\begin{proof}
Note that $X$ can be rewritten as
\begin{eqnarray*}
X &=&\{x\in \mathbb{R}^{n}:(\overline{a}_{j}^{0})^{\top
}x-b_{j}+\displaystyle \sum\nolimits_{l=1}^{q_{j}}v_{j}^{l}(\overline{a}_{j}^{l})^{\top
}x)\geq 0\mbox{ for all } \\
&&\Vert (v_{j}^{1},\dots v_{j}^{q_{j}})\Vert \leq 1,b_{j}\in \lbrack
\underline{b}_{j},\overline{b}_{j}],\,j=1,\ldots ,p\medskip \big)\} \\
&=&\left\{ x\in \mathbb{R}^{n}:(\overline{a}_{j}^{0})^{\top }x-\overline{b}%
_{j}-\Vert \big((\overline{a}_{j}^{1})^{\top }x,\ldots ,(\overline{a}%
_{j}^{q_{j}})^{\top }x\big)\Vert \geq 0,\,j=1,\ldots ,p\right\} .
\end{eqnarray*}%
The conclusion will follow from Theorem \ref{th:1} if we show that
\begin{equation*}
N(X,\overline{x})=\left\{ u\in \mathbb{R}^{n}:\exists \mu _{j}\geq 0,\Vert
w\Vert \leq 1\mbox { s.t.
}-u^{\top }\overline{x}=\sum_{j=1}^{p}\mu _{j}\overline{b}_{j}\mbox{ and }%
-u-\sum_{j=1}^{p}\mu _{j}(\overline{a}_{j}^{0}-y_{j})=0_{m}\right\} .
\end{equation*}%
To see this, let $u\in N(X,\overline{x})$. Then, $\overline{x}$ is a
solution of the following convex optimization problem:
\begin{equation*}
\min \left\{ -u^{\top }x:(\overline{a}_{j}^{0})^{\top }x-\overline{b}%
_{j}-\Vert \big((\overline{a}_{j}^{1})^{\top }x,\ldots ,(\overline{a}%
_{j}^{q_{j}})^{\top }x\big)\Vert \geq 0,j=1,\ldots ,p\right\} .
\end{equation*}%
As the strict feasibility condition \eqref{SlaterEllipsoid} holds, by the
Lagrangian duality, there exist $\mu _{j}\geq 0,\,j=1,\ldots ,p,$ such that
\begin{equation*}
-u^{\top }\overline{x}=\min_{x\in \mathbb{R}^{n}}\left\{ (-u^{\top
}x)+\sum_{j=1}^{p}\mu _{j}(-(\overline{a}_{j}^{0})^{\top }x+\overline{b}%
_{j}+\Vert \big((\overline{a}_{j}^{1})^{\top }x,\ldots ,(\overline{a}%
_{j}^{q_{j}})^{\top }x\big)\Vert )\right\} .
\end{equation*}%
As $\overline{x}\in X$, this implies that $\mu _{j}\big((\overline{a}%
_{j}^{0})^{\top }\overline{x}-\overline{b}_{j}-\Vert \big((\overline{a}%
_{j}^{1})^{\top }\overline{x},\ldots ,(\overline{a}_{j}^{q_{j}})^{\top }%
\overline{x}\big)\Vert \big)=0,j=1,\ldots ,p,$ and so, the function $%
h(x):=(-u^{\top }x)+\sum_{j=1}^{p}\mu _{j}(-(\overline{a}_{j}^{0})^{\top }x+%
\overline{b}_{j}+\Vert \big((\overline{a}_{j}^{1})^{\top }x,\ldots ,(%
\overline{a}_{j}^{q_{j}})^{\top }x\big)\Vert )$ attains its minimum at $%
\overline{x}$ and $\min_{x\in \mathbb{R}^{n}}h(x)=-u^{\top }\overline{x}$.
This implies that $0_{n}\in \partial h(\overline{x})$, and so, there exists $%
w\in \mathbb{R}^{n}$ with $\Vert w\Vert \leq 1$ such that
\begin{equation*}
-u^{\top }\overline{x}=\sum_{j=1}^{p}\mu _{j}\overline{b}_{j}\mbox{ and }%
-u-\sum_{j=1}^{p}\mu _{j}\overline{a}_{j}^{0}+\sum_{j=1}^{p}\mu
_{j}y_{j}=0_{n},
\end{equation*}%
where $y_{j}=\big((\overline{a}_{j}^{1})^{\top }w,\ldots ,(\overline{a}%
_{j}^{q_{j}})^{\top }w\big)^{\top }$. Then, we have
\begin{equation*}
N(X,\overline{x})\subset \left\{ u\in \mathbb{R}^{n}:\exists \mu _{j}\geq
0,\Vert w\Vert \leq 1\mbox { s.t. }-u^{\top }\overline{x}=\sum_{j=1}^{p}\mu
_{j}\overline{b}_{j}\mbox{ and }-u-\sum_{j=1}^{p}\mu _{j}(\overline{a}%
_{j}^{0}-y_{j})=0_{n}\right\} .
\end{equation*}%
To see the reverse inclusion, let $u\in \mathbb{R}^{n}$ be such that $%
-u^{\top }\overline{x}=\sum_{j=1}^{p}\mu _{j}\overline{b}_{j}\mbox{ and }%
-u-\sum_{j=1}^{p}\mu _{j}(\overline{a}_{j}^{0}-y_{j})=0_{n}$ for some $\mu
_{j}\geq 0,$ $j=1,...,p,$ and $\Vert w\Vert \leq 1$. Then, for all $x\in X$,
\begin{equation*}
u^{\top }(x-\overline{x})=\sum_{j=1}^{p}\mu _{j}\overline{b}%
_{j}-\sum_{j=1}^{p}\mu _{j}(\overline{a}_{j}^{0}-y_{j})^{\top }x\leq
\sum_{j=1}^{p}\mu _{j}(-(\overline{a}_{j}^{0})^{\top }x+\overline{b}%
_{j}+\Vert \big((\overline{a}_{j}^{1})^{\top }x,\ldots ,(\overline{a}%
_{j}^{q_{j}})^{\top }x\big)\Vert )\leq 0.
\end{equation*}%
Thus, $u\in N(X,\overline{x})$ and so, the conclusion follows.
\end{proof}

\smallskip

The above robust solution characterization under the constraint ellipsoidal
data uncertainty shows that one can determine whether a robust feasible
point is a robust weakly efficient solution point or not by solving finitely
many second order cone systems as long as the strict robust feasibility
condition \eqref{SlaterEllipsoid} is satisfied.\bigskip

Finally, it should be noted that there are other approaches in defining
robust solutions for uncertain multiobjective optimization when the data
uncertainty $\mathcal{U}\subset \mathbb{R}^{m\times n}$ in the objective
matrix is a columnwise objective data uncertainty, that is, $\mathcal{U}%
=\prod_{i=1}^{m}\mathcal{U}_{i}$ where $\mathcal{U}_{i}\subset \mathbb{R}%
^{n} $. In this case, one can define a robust solution of the uncertain
multi-objective optimization problem as the solution of the following
deterministic multiobjective optimization problem
\begin{equation*}
\text{V-}\min \left\{ \left( \max_{c_{1}\in \mathcal{U}_{1}}c_{1}^{\top
}x,\ldots ,\max_{c_{m}\in \mathcal{U}_{m}}c_{m}^{\top }x\right) :a_{j}^{\top
}x\geq b_{j},\;\forall (a_{j},b_{j})\in \mathcal{V}_{j},\,j=1,\ldots
,p\right\} .
\end{equation*}%
This approach has been recently examined in the paper \cite{GJLP} for
uncertain multiobjective optimization with semi-infinite constraints under
columnwise objective data uncertainty.

\section*{Acknowledgments}

M.A. Goberna would like to thank the coauthors of this paper for their
hospitality during his stay in June/July 2013 at the School of Mathematics
and Statistics of the University of New South Wales in Sydney.

\end{document}